\documentclass[a4paper,12pt]{amsart}
\usepackage{amsmath}
\usepackage{amsthm}
\usepackage{amssymb}
\usepackage{mathtools}
\usepackage{fullpage}
\usepackage{bbm}
\usepackage{verbatim}
\usepackage[colorlinks]{hyperref}
\numberwithin{equation}{section}
\newtheorem{theorem}{Theorem}[section]
\newtheorem{lemma}[theorem]{Lemma}
\newtheorem{corollary}[theorem]{Corollary}
\newtheorem{conjecture}{Conjecture}[section]

\numberwithin{equation}{section}
\begin{document}

\title{On the sum of the reciprocals of the differences
between consecutive primes}

\author{Nianhong Zhou}
\address{Department of Mathematics\\
East China Normal University\\
500 Dongchuan Road, Shanghai 200241, PR China}
\email{nianhongzhou@outlook.com}
\begin{abstract}
Let $p_n$ denote the $n$-th prime number, and let $d_n=p_{n+1}-p_{n}$. Under the Hardy--Littlewood prime-pair conjecture, we prove
\begin{align*}
\sum_{n\le X}\frac{\log^{\alpha}d_n}{d_n}\sim\begin{cases}
\quad\frac{X\log\log\log X}{\log X}~\qquad\quad~ &\alpha=-1,\\
\frac{X}{\log X}\frac{(\log\log X)^{1+\alpha}}{1+\alpha}\qquad &\alpha>-1,
\end{cases}
\end{align*}
and establish asymptotic properties for some series of $d_n$ without the Hardy--Littlewood prime-pair conjecture.
\end{abstract}
\maketitle


\section{Introduction} \label{section 1}

Let $p_n$ denote the $n$-th prime number, and let $d_n=p_{n+1}-p_{n}$. In \cite{MR1420202}, Erd\"{o}s and Nathanson show that for $c>2$,
\begin{equation}\label{11}
\sum_{n=3}^{\infty}\frac{1}{d_nn(\log\log n)^c}<+\infty.
\end{equation}
The authors give a heuristic argument explaining why the series must diverge for $c=2$. We will prove the above \eqref{11} by some conclusions of the sieve method.

Let $\mathcal{H}=\{0,h_1,\dots,h_{k-1}\}$ be a set of $k(k\ge 2)$ distinct integers satisfying $0<h_1<h_2<\dots<h_{k-1}$ and not covering all residue classes to any prime modulus. Also, denote
\begin{equation*}
\pi(x; \mathcal{H})=\#\{n\in\mathbb{N}:n+h_{k-1}\le x,n,n+h_1,\dots,n+h_{k-1}~\text{are all primes}\}.
\end{equation*}
The Hardy--Littlewood prime $k$-tuple conjecture is that, for $X\rightarrow+\infty$,
\begin{equation*}
\pi(X; \mathcal{H})=\mathfrak{S}(\mathcal{H})\frac{X}{\log^{k}X}\left(1+o(1)\right),
\end{equation*}
where the singular series
\begin{equation*}
\mathfrak{S}(\mathcal{H})=\prod_{p}\left(1-\frac{v_{\mathcal{H}}(p)}{p}\right)\left(1-\frac{1}{p}\right)^{-k},
\end{equation*}
with $p$ running through all the primes and
\[v_{\mathcal{H}}(p)=\#\{m~{(\bmod~p)}: m(m+h_1)\dots(m+h_{k-1})\equiv 0\bmod p\}.\]
We will also need the following well-known sieve bound, for $X$ sufficiently large,
\begin{equation}\label{14}
\pi(X; \{0, h, d\})\le 2^3\times 3! \mathfrak{S}(\{0, h, d\})\frac{X}{\log^{3}X}\left(1+o(1)\right),
\end{equation}
when $\mathfrak{S}(\{0, h, d\})\neq 0$. (See
Iwaniec and Kowalski's excellent monograph \cite{MR2061214}.)

To prove our main theorem, we will need the following Hardy--Littlewood prime-pair conjecture.
\begin{conjecture}\label{conj1}Let $X$ be sufficiently large and $d\ll \log X$ be a natural number. Then
\begin{equation}
\pi(X,\{0,d\})=\mathfrak{S}(\{0,d\})\frac{X}{\log^{2}X}\left(1+o(1)\right),
\end{equation}
where
\begin{align*}
\mathfrak{S}(\{0,d\})=\begin{cases}2\prod\limits_{p>2}\left(1-\frac{1}{(p-1)^2}\right)\prod\limits_{p|d,p>2}\left(\frac{p-1}{p-2}\right)&\text{if d is even},\\
\qquad\qquad\qquad 0\qquad\qquad\qquad&\text{if d is odd}.
\end{cases}
\end{align*}
with the product extending over all primes $p>2$.
\end{conjecture}
Our main result can be summarized as follows under the above conjecture.
\begin{theorem}\label{them1}Assume that the Hardy--Littlewood prime-pair conjecture holds for all sufficiently large $X$. Then we have
\begin{align*}
\sum_{n\le X}\frac{\log^{\alpha}d_n}{d_n}\sim\begin{cases}\frac{X}{\log X}\frac{(\log\log X)^{1+\alpha}}{1+\alpha}\qquad &\alpha>-1,\\
\quad\frac{X\log\log\log X}{\log X}\qquad &\alpha=-1.
\end{cases}
\end{align*}
\end{theorem}

Letting $\alpha=0$ in above theorem and using Abel's summation formula, one can obtain the following corollary.
\begin{corollary}\label{cor1}Let $X$ be sufficiently large. Then
\begin{align*}
\sum_{3\le n\le X}\frac{1}{d_nn(\log\log n)^c}=\begin{cases}\qquad\quad \gamma_c+o(1)\qquad &c>2,\\
~\log\log\log X\left(1+o(1)\right)\quad&c=2,\\
~\frac{(\log\log X)^{2-c}}{2-c}\left(1+o(1)\right)\qquad &c<2,
\end{cases}
\end{align*}
where $\gamma_c$ is a constant.
\end{corollary}
Without the Hardy--Littlewood prime-pair conjecture, using the same idea one can obtain the following result.
\begin{theorem}\label{them2}Let $X$ be sufficiently large. Then
\begin{align*}
\sum_{n\le X}\frac{\log^{\alpha}d_n}{d_n}\ll \begin{cases}\frac{X}{\log X}\frac{(\log\log X)^{1+\alpha}}{1+\alpha}\qquad &\alpha>-1,\\
\quad\frac{X\log\log\log X}{\log X}\qquad &\alpha=-1.
\end{cases}
\end{align*}
\end{theorem}

Similar to Corollary \ref{cor1}, one can obtain the following corollary.
\begin{corollary}\label{cor2}Let $X$ be sufficiently large. Then
\begin{align*}
\sum_{3\le n\le X}\frac{1}{d_nn(\log\log n)^c}=\begin{cases}\qquad \gamma_c+o(1)~\qquad &c>2,\\
O\left(\log\log\log X\right)~\qquad &c=2,\\
O\left(\frac{(\log\log X)^{2-c}}{2-c}\right)\qquad &c<2,
\end{cases}
\end{align*}
where $\gamma_c$ is a constant.
\end{corollary}

\section{Basic Lemma}\label{section 2}
To prove Theorem \ref{them1}, we need the following lemmas.
\begin{lemma}\label{lem21}{\rm (See~\cite[Proposition~1]{MR1299655})}. Let $X$ be sufficiently large. Then
\begin{align*}
\sum_{d\le X}\mathfrak{S}(\{0, d\})-X+\frac{\log X}{2}\ll \log^{\frac{2}{3}} X.
\end{align*}
\end{lemma}
As a special case of \cite[Lemma~2]{MR3415646}, we have
\begin{lemma}\label{lem22} Let $d$ be an even integer. Then
\begin{align*}
\sum_{h=1}^{d-1}\mathfrak{S}(\{0, h, d\})=\mathfrak{S}(\{0, d\})d(1+o_d(1)).
\end{align*}
\end{lemma}
The following lemma is important in this paper.
\begin{lemma}\label{lem23}Let $f(x)\in \mathcal{C}^1[2,+\infty)$ be strictly monotonically decreasing to $0$, and $\int_{2}^{\infty}f(t)\,dt$ divergence. Also, let $X$ sufficiently large, $y=o(\log X)$ and $y\gg \log\log X$.

(a)  Using Conjecture \ref{conj1}, we have
\[
\sum_{\substack{d_n\le y\\ p_{n+1}\le X}}f(d_n)\sim \frac{X}{\log X}\int_{2}^{y}f(t)\,dt.
\]

(b). Without using Conjecture \ref{conj1}, we have
\[
\sum_{\substack{d_n\le y\\ p_{n+1}\le X}}f(d_n)\ll \frac{X}{\log X}\int_{2}^{y}f(t)\,dt.
\]
\end{lemma}
\begin{proof}
The proof of parts (a) and (b) are essentially the same. Therefore, we prove part (a) only. Firstly, we have
\begin{align*}
\sum_{\substack{d_n\le y \\ p_{n+1}\le X}}f(d_n)&=\sum_{d\le y}f(d)\sum_{\substack{d_n=d\\ p_{n+1}\le X}}1\\
&=\sum_{d\le y}f(d)\pi(X; \{0, d\})+\sum_{d\le y}f(d)\left(\sum_{\substack{d_n=d\\p_{n+1}\le X}}1-\pi(X; \{0, d\})\right).
\end{align*}
By the inclusion-exclusion principle, it is easy to see that
\[\pi(X; \{0, d\})-\sum_{h=1}^{d-1}\pi(X; \{0, h, d\})\le \sum_{\substack{d_n=d\\ p_{n+1}\le X}}1\le \pi(X; \{0, d\}).\]
Hence
\begin{align}\label{sfa}
\sum_{\substack{d_n\le y\\p_{n+1}\le X}}f(d_n)=\sum_{d\le y}f(d)\pi(X; \{0, d\})+O\left(\sum_{d\le y}f(d)\sum_{h=1}^{d-1}\pi(X; \{0, h, d\})\right).
\end{align}
Combining (\ref{14}) with Lemma \ref{lem22}, we see that the error term in (\ref{sfa}) is
\[\ll \frac{X}{\log^3 X}\sum_{d\le y}f(d)\sum_{h=1}^{d-1}\mathfrak{S}(\{0, h, d\})\ll\frac{X}{\log^3 X}\sum_{d\le y}f(d)d\mathfrak{S}(\{0, d\}).\]
Using Abel's summation formula, noting that $f(x)\in \mathcal{C}^1[2,+\infty)$ is strictly monotonically decreasing to $0$ and $y\gg \log\log X$, we have
\begin{align*}
\sum_{d\le y}f(d)d\mathfrak{S}(\{0, d\})=\int_{2}^{y}f(x)x\,d\left(\sum_{d\le x}\mathfrak{S}(\{0, d\})\right)\ll\int_{2}^{y}f(x)x\,dx.
\end{align*}
Together with \eqref{conj1}, we have
\begin{align}\label{26}
\sum_{\substack{d_n\le y\\ p_{n+1}\le X}}f(d_n)=\frac{X}{\log^2 X}\left(1+o(1)\right)\sum_{d\le y}f(d)\mathfrak{S}(\{0, d\})+O\left(\frac{X}{\log^3 X}\int_{2}^{y}f(x)x\,dx\right).
\end{align}
Combining Lemma \ref{lem21} and using Abel's summation formula again, we
obtain
\begin{align}\label{27}
\sum_{d\le y}f(d)\mathfrak{S}(\{0, d\})&=\int_{2}^yf(x)\,d\left(\sum_{d\le x}\mathfrak{S}(\{0, d\})\right)\nonumber\\
&=\int_{2}^yf(x)\,dx+O(1)+O(f(y)\log y)+O\left(\int_{2}^y\frac{f(x)}{x}\,dx\right).
\end{align}
Since
$$\int_{2}^y\frac{f(x)}{x}\,dx=f(y)\log y-\int_{2}^yf'(x)\log x\,dx+O(1)$$
and $-\int_{2}^yf'(x)\log x\,dx>0$ by the assumption on $f$, hence by \eqref{26} and \eqref{27} we have
\[
\sum_{\substack{d_n\le y\\ p_{n+1}\le X}}f(d_n)=\frac{X}{\log^2 X}\left(\left(1+o(1)\right)\int_{2}^yf(x)\,dx+O\left(\int_{2}^y\frac{f(x)}{x}\,dx+\frac{\int_{2}^{y}f(x)x\,dx}{\log X}\right)\right).
\]
By using L'Hospital's rule, we get
\[
\lim_{y\rightarrow+\infty}\frac{\int_{2}^yf(x)x^{-1}\,dx}{\int_{2}^yf(x)\,dx}=0~~\text{and}~~\lim_{y\rightarrow+\infty}\left|\frac{\int_{2}^yf(x)x\,dx}{y\int_{2}^yf(x)\,dx}\right|\le 1.
\]
Hence
\[
\sum_{\substack{d_n\le y\\ p_{n+1}\le X}}f(d_n)=\frac{X}{\log^2 X}\left(1+o(1)+O\left(\frac{y}{\log X}\right)\right)\int_{2}^yf(x)\,dx.
\]
On noting that $y=o(\log X)$, we obtain the proof of part (a).
\end{proof}

\begin{lemma}\label{lem24}Let $f(x)\in \mathcal{C}^1[2,+\infty)$ be strictly monotonically decreasing to $0$, and $\int_{2}^{\infty}f(t)\,dt$ divergence. Also, let $X$ sufficiently large and $\log^{\frac{1}{2}}X\le x<y\le \log^2 X$. Then we have
\[
\sum_{\substack{x<d_n\le y\\ p_{n+1}\le X}}f(d_n)\ll \frac{X}{\log^2 X}\left(\int_{x}^yf(t)\,dt+f(x)\log \log X\right).
\]
\end{lemma}
\begin{proof}
Since $f(x)$ is strictly monotonically decreasing and $\log^{\frac{1}{2}}X\le x<y\le \log^2 X$, we have
\begin{align*}
\sum_{\substack{x<d_n\le y\\ p_{n+1}\le X}}f(d_n)&=\sum_{x<d\le y}f(d)\sum_{\substack{d_n=d\\ p_{n+1}\le X}}1\le \sum_{x<d\le y}f(d)\pi(X; \{0, d\})\\
&\ll\frac{X}{\log^2 X} \sum_{x<d\le y}f(d)\mathfrak{S}(\{0, d\})\ll\frac{X}{\log^2 X}\left(\int_{x}^yf(t)\,d\left(t+O(\log t)\right)\right)\\
&\ll\frac{X}{\log^2 X}\left(\int_{x}^yf(t)\,dt+f(x)\log y+\int_{x}^y\left|f'(t)\right|\log t\,dt\right)\\
&\ll\frac{X}{\log^2 X}\left(\int_{x}^yf(t)\,dt+f(x)\log \log X\right).
\end{align*}
This completes the proof of the lemma.
\end{proof}

\section{The proof of main theorem}
Let $y=\log X(\log\log X)^{-1}$ and $f(t)=t^{-1}\log^{\alpha}t$ $(\alpha\ge -1)$. Using Lemma \ref{lem23} and Lemma \ref{lem24}, we have
\begin{align*}
\sum_{p_{n+1}\le X}f(d_n)=&\sum_{\substack{d_n\le y\\p_{n+1}\le X}}f(d_n)+\sum_{\substack{p_{n+1}\le X\\y<d_n\le \log X}}f(d_n)+\sum_{\substack{d_n>\log X\\p_{n+1}\le X}}f(d_n)\\
=&\frac{X}{\log^2X}(1+o(1))\int_{2}^yf(x)\,dx\\
&+O\left(\frac{X\int_{y}^{\log X}f(t)\,dt}{\log^2X}+\frac{X\log \log X}{\log^2X}f(y)\right)+O\left(\frac{Xf(\log X)}{\log X}\right).
\end{align*}
Substituting the values of $f$ and $y$ into the above equation, we obtain
\begin{align*}
\sum_{p_{n+1}\le X}\frac{\log^{\alpha}d_n}{d_n}=&\frac{X}{\log^2X}(1+o(1))\int_{\log 2}^{\log \left(\log X(\log\log X)^{-1}\right)}u^{\alpha}\,du\\
&+\frac{X}{\log^2X}\left(o(1)+O\left((\log\log X)^{\alpha}\log\log\log X\right)\right).
\end{align*}
Hence we get
\begin{align}\label{31}
\sum_{p_{n+1}\le X}\frac{\log^{\alpha}d_n}{d_n}\sim\begin{cases}\frac{X}{\log^2 X}\frac{(\log\log X)^{1+\alpha}}{1+\alpha}\qquad &\alpha>-1,\\
\quad\frac{X\log\log\log X}{\log^2 X} &\alpha=-1.
\end{cases}
\end{align}
By prime number theorem, the maximum integer $n$ satisfying $p_{n+1}\le X$ is $X\log X(1+o(1))$ and  substituting these values into \eqref{31} above completes the proof of the theorem.

\section*{Acknowledgment}
The author would like to thank the anonymous referees and the editors for their very helpful
comments and suggestions. The author also thank Min-Jie Luo for offering many useful suggestions and
help.

\end{document}